\documentclass[10pt]{article}

\font\smallit=cmti10

\makeatletter
\def\blfootnote{\xdef\@thefnmark{}\@footnotetext}
\makeatother

\usepackage{amssymb,latexsym,amsmath,epsfig,amsthm,mathtools,xcolor,lipsum}
\usepackage{scalerel,nicefrac}
\usepackage[utf8]{inputenc}

\makeatletter

\renewcommand\section{\@startsection {section}{1}{\z@}
{-30pt \@plus -1ex \@minus -.2ex}
{2.3ex \@plus.2ex}
{\normalfont\normalsize\bfseries\boldmath}}

\renewcommand\subsection{\@startsection{subsection}{2}{\z@}
{-3.25ex\@plus -1ex \@minus -.2ex}
{1.5ex \@plus .2ex}
{\normalfont\normalsize\bfseries\boldmath}}

\renewcommand{\@seccntformat}[1]{\csname the#1\endcsname. }

\makeatother

\newtheorem{theorem}{Theorem}
\newtheorem{lemma}{Lemma}
\newtheorem{proposition}{Proposition}

\theoremstyle{definition}

\newtheorem{remark}{Remark}

\usepackage{hyperref}
\usepackage{tcolorbox}

\begin{document}

\begin{center}
\uppercase{\bf On the minimal number of solutions\\
of the equation $\boldsymbol{\phi(\lowercase{n+k}) = M \, \phi(\lowercase{n})}$, $\boldsymbol{M=1}$, $\boldsymbol{2}$}
\vskip 20pt
{\bf Matteo Ferrari}\\
{\smallit Dipartimento di Scienze Matematiche ``G.L. Lagrange'', Politecnico di Torino Corso Duca degli Abruzzi 24, 10138, Torino, Italy}\\
{\tt matteo.ferrari@polito.it}\\ 
\vskip 10pt
{\bf Lorenzo Sillari}\\
{\smallit  Scuola Internazionale Superiore di Studi Avanzati (SISSA) \\ Via Bonomea 265, 34136, Trieste, Italy}\\
{\tt lsillari@sissa.it}\\ 
\end{center}
\vskip 30pt

\centerline{\bf Abstract} We fix a positive integer $k$ and look for solutions of the equations $\phi (n+k) = \phi (n)$ and $\phi(n+k) = 2 \phi (n)$. We prove that Fermat primes can be used to build five solutions for the first equation when $k$ is even and five for the second one when $k$ is odd. These results hold for $k \le 2 \cdot 10^{100}$. We also show that for the second equation with even $k$ there are at least three solutions for $k \le 4 \cdot 10^{58}$. Our work increases the previous minimal number of known solutions for both equations.
\noindent

\blfootnote{\hspace{-0.55cm} 
\textit{Keywords} Euler's Phi function\\[5pt]
The numerical computations supporting the results in our paper have been performed using Magma, and are available upon request to the first author.}
\pagestyle{myheadings} 
\thispagestyle{empty} 
\baselineskip=12.875pt 
\vskip 30pt

\begin{center}
\section{Introduction}
\end{center}

Euler's Phi function $\phi(n)$ counts the number of positive integers less than or equal to $n$ that are coprime with $n$. In 1956, Sierpi\'{n}ski considered in \cite{Sierpinski1956} the equation
\begin{equation}\label{equation:M:1}
\phi (n+k) = \phi (n),
\end{equation}
and proved that for every positive integer $k$ (with no upper bound) there is at least one solution. The result was improved by Schinzel in \cite{Schinzel1958}, showing that \eqref{equation:M:1} has at least two solutions for every positive integer $k \le 8 \cdot 10^{47}$ and later by Schinzel and Wakulicz in \cite{SchinzelWakulicz1959} for $k \le 2 \cdot 10^{58}$. In \cite{Holt2003}, Holt extends the upper bound to $1.38 \cdot 10^{26595411}$ for even $k$. All these results where obtained exhibiting explicit solutions built proving the existence (or assuming it) of prime numbers satisfying certain conditions. Graham, Holt and Pomerance further discussed the problem, showing that for even $k$ if we take $j$ and $r$ such that $j$ and $j+k$ have the same prime factors, and the numbers $j/g r +1$, $(j+k)/g r +1$ \scalerel{(}{\strut} with $g= (j,j+k)$ \scalerel{)}{\strut} are both primes not dividing $j$, then 
\[
n = j \left( \frac{j+k}{g}r +1 \right) 
\]
is a solution of \eqref{equation:M:1} (cf. \cite{GrahamHoltPomerance1999}, Theorem 1). Despite the lack of general results, numerical evidence suggests that \eqref{equation:M:1} has infinitely many solutions for every fixed $k$. This can be proven for some special values of $k$ (cf. \cite{Sungjing2021}). \\
\vspace{.3cm}

In the same spirit of Sierpi\'{n}ski, Makowski considered in \cite{Makowski1974} the equation
\begin{equation}\label{equation:M:2}
    \phi (n+k) = 2 \phi (n),
\end{equation}
finding one solution for all fixed $k$. More recently, the same equation was studied by Hasanalizade in \cite{Hasanalizade2021}. He proves that \eqref{equation:M:2} has at least two solutions for all $k \le 4 \cdot 10^{58}$ and at least three solutions for some odd $k \le 4 \cdot 10^{58}$ (cf. \cite{Hasanalizade2021}, Theorem 1). Also he gives, in Lemma 1 of \cite{Hasanalizade2021}, a modified version of Theorem 1 from Graham, Holt and Pomerance, valid for odd $k$ multiple of $3$ and equation \eqref{equation:M:2}.
\vspace{.3cm}

In our paper we find more solutions for \eqref{equation:M:1} and \eqref{equation:M:2}. These solutions are obtained in two ways. First we consider Fermat prime numbers and show that for each Fermat prime it is possible to build a solution for \eqref{equation:M:1} if $k$ is even and for \eqref{equation:M:2} if $k$ is odd. Numerically, we show that these solutions can be actually built for $ k \le 2 \cdot 10^{100}$. This allows to find five different solutions. Second, we exhibit a new solution for equation \eqref{equation:M:2} and even $k$. Together with the results of Hasanalizade, this gives our main theorem.
\begin{theorem}\label{main}
Equation \eqref{equation:M:2} has at least three solutions for all $k \le 4 \cdot 10^{58}$.
\end{theorem}
Finally we provide some way of building particular solutions if certain conditions are satisfied.
\vspace{.5cm}

\begin{center}
    \section{The equation $\boldsymbol{\phi(n) = \phi(n+k)}$}
\end{center}
Denote by $F_m$ the $m$-th Fermat number
\[
F_m = 2^{2^m}+1.
\]
It is known that for $m=0, \dots ,4$, $F_m$ is prime and that $F_m$ is composite for all $5 \le m \le 32$ and for some bigger value of $m$. It is not known whether there exist other values of $m$ such that $F_m$ is prime.\\
Let $a$ and $b$ be positive integers. If the prime factors of $a$ are contained in the prime factors of $b$ we write $a |^* b$. If $a |^* b$, then it is easy to see that
\[
\phi (ab) = a \phi(b).
\]
Moreover, $a$ has the same prime factors of $b$ if and only if $a |^* \, b $ and $b |^* \, a$. In this case, we have that
\[
a \phi(b) = b \phi (a).
\]
We prove the following result.

\begin{theorem}\label{even:M:1}
    Equation \eqref{equation:M:1} has at least six solutions for all even $k \le 2 \cdot 10^{100}$.
\end{theorem}
\begin{proof}
We build a solution for each $F_m$ prime, noticing that for even $k$, $(F_m-1)|^* k$.\\
\textbf{Case 1: $\boldsymbol{(F_m,k)=1.}$} If $F_m$ and $k$ are coprime, then
\[
n = (F_m -1) k = 2^{2^m} k
\]
is a solution of \eqref{equation:M:1}, in fact 
\begin{align*}
    \phi (n + k ) &= \phi (F_m k ) = \phi(F_m) \phi(k) = (F_m -1) \phi (k), \\
    \phi(n) &= \phi ((F_m-1)k) = (F_m-1) \phi(k).
\end{align*}
Note that in this case we have no upper bound on the values of $k$.\\
\textbf{Case 2: $\boldsymbol{F_m | k}$}. Assume that there exists a positive integer $r$ such that $(F_m -1)r +1$ and $F_m r +1$ are both primes and do not divide $k$. Then 
\[
n = (F_m-1)(F_m r+1)k
\]
is a solution of \eqref{equation:M:1}, in fact
\begin{align*}
    \phi(n+k) &= \phi \Big(\big[ (F_m-1)F_m r + F_m \big] k \Big) \\
    &= \phi \Big( F_m \big((F_m-1)r+1 \big)k \Big) \\
    &= F_m (F_m-1) r \phi(k), \\
    \phi(n) &= \phi \Big( (F_m-1)(F_mr+1)k \Big) \\
    &= F_m (F_m-1) r \phi(k).
\end{align*}
For the following values of $r$ and $m$, we have that $(F_m-1)r+1$ and $F_mr+1$ are both prime
\[
\begin{array}{c|l}
     m & r \\
     \hline
     \hspace{0.1cm} \vspace{-8pt} \\
     0 & 10^{100} + 9760\\
     1 & 10^{100} + 60128\\
     2 & 10^{100} + 150326\\
     3 & 10^{100} + 51326\\
     4 & 10^{100} + 14786\\
\end{array}
\]
In particular, taking $ k \le 2 \cdot 10^{100}$, they are surely prime with $k$. Solutions obtained from different choices of $m$ and $r$ are different since they are exactly divisible by different powers of $2$, thus we have at least five solutions coming from the Fermat primes.\\
To these we must add the previously known solutions. In particular the solution of Sierpi\'{n}ski might coincide with one of our solutions for some $k$, while the solution from Schinzel differs from the ones we provided. This brings the minimal number of known solutions to six for all even $k \le 2 \cdot 10^{100}$.
\end{proof}

\begin{remark}\label{rem:Dickson}
    It is a well-known conjecture by Dickson that for any fixed $a$, $b$ there exist infinitely many positive integers $r$ such that $ar+1$ and $br+1$ are both prime. Following Graham, Holt and Pomerance, we denote this property with $\mathcal{P}(a,b)$. Assuming $\mathcal{P}(F_m,F_m-1)$ for $m =0, \dots ,4$, we can remove the upper bound from $k$ and we obtain that \eqref{equation:M:1} has at least six solutions for all even $k$.
\end{remark}

\begin{remark}\label{rem:infinite}
    On the other side, assume that $F_m$ divides $k$, for some fixed $m$. Then, for any $r$ such that $(F_m-1)r+1$ and $F_mr+1$ are both prime and do not divide $k$, we find a solution $n$ depending on $r$ and different values of $r$ yield different solutions. As a consequence, if $\mathcal{P}(F_m,F_m-1)$ is true, then equation \eqref{equation:M:1} has an infinite number of solutions for all even $k$ multiple of $F_m$.
\end{remark}

\begin{center}
\section{The equation $\boldsymbol{\phi(n+k) = 2\phi(n)}$}
\end{center}

Our main results concern the number of solutions of \eqref{equation:M:2}. As for equation \eqref{equation:M:1}, there is substantial difference in the solutions that can be found for even $k$ and odd $k$, but the roles are here inverted. We have in fact an analogous of Theorem \ref{even:M:1} (with a similar proof), but for odd $k$ instead of even $k$.

\begin{theorem}\label{odd:M:2}
    Equation \eqref{equation:M:2} has at least five solutions for all odd $k \le 2 \cdot 10^{100}$.
\end{theorem}
\begin{proof}
We build a solution for each $F_m$ prime, noticing that $(F_m-1,k)=1$, and that
\[
F_m -1 = 2 \phi (F_m -1).
\]
\textbf{Case 1: }$ \boldsymbol{(F_m,k)=1.}$ If $F_m$ and $k$ are coprime, then
\[
n = (F_m -1) k = 2^{2^m} k
\]
is a solution of \eqref{equation:M:2}.
Again, in this case we have no upper bound on the values of $k$ for which we find solutions.\\
\textbf{Case 2: } $\boldsymbol{F_m | k}$. Assume that there exists a positive integer $r$ such that $(F_m -1)r +1$ and $F_m r +1$ are both primes and do not divide $k$. Then 
\[
n = (F_m-1)(F_m r+1)k
\]
is a solution of \eqref{equation:M:1}.
Reasoning as in Theorem \ref{even:M:1}, we have five solutions for all odd $k \le 2 \cdot 10^{100}$.
\end{proof}

\begin{remark}
The considerations we made in Remark \ref{rem:Dickson} and \ref{rem:infinite} are valid also in this case, replacing even $k$ with odd $k$. In particular for $m=0$, Remark \ref{rem:infinite} applied to equation \eqref{equation:M:2} has as a consequence the second part of Lemma 1 in \cite{Hasanalizade2021}.
\end{remark}

Regarding even $k$, Makowski found in \cite{Makowski1974} a solution to equation \eqref{equation:M:2} for all even $k$ given by $n = k$. Another solution was found by Hasanalizade in \cite{Hasanalizade2021}, building a sequence of primes. We recall here the construction. Take a sequence of primes $3=p_1 < p_2 < \dots < p_m$ satisfying for all $i = 2, \dots, m$
\begin{itemize}
\item $(p_i-2) | \, \prod\limits_{j \le i-1} p_j$,
\item $(p_i-1) |^* \, 2 \prod\limits_{j \le i-1} p_j$.
\end{itemize}
and such that $\prod_{j} p_j$ does not divide $k$. Let $p_l$ be the smallest prime of the sequence such that $(p_l,k)=1$. Then $n = p_l/(p_l-2)k$ is a solution of equation \eqref{equation:M:2}. Building the explicit sequence of primes $\{3$, $5$, $7$, $17$, $19$, $37$, $97$, $113$, $257$, $401$, $487$, $631$, $971$, $1297$, $1801$, $19457$, $22051$, $28817$, $65537$, $157303$, $160001\}$, we have one solution different from $n=k$ for all even $k \le 4 \cdot 10^{58}$.\\
Now we find a third solution for even $k$.

\begin{theorem}\label{even:M:2}
Equation \eqref{equation:M:2} has at least three solutions for all even $k \le 4 \cdot 10^{58}$.
\end{theorem}
\begin{proof}
Assume that there exists a sequence of prime numbers $2 = p_1 < p_2 < \ldots < p_m$ and of positive integers $a_2 < \ldots < a_m$ such that for all $i = 2, \dots,m$
\begin{itemize}
    \item $p_i = 2a_i +1$,
    \item $a_i |^* \, \prod\limits_{j \le i-1} p_j$,
    \item $(a_i +1) | \, \prod\limits_{j \le i-1} p_j$.
\end{itemize}
and such that $\prod_j p_j$ does not divide $k$. Let $p_l$ be the smallest prime such that $(p_l,k)=1$. Then $n = a_l/(a_l+1)k$ is a solution of \eqref{equation:M:2}, in fact
\begin{align*}
\phi(n+k) & = \phi\left( \frac{2a_{l}+1}{a_{l}+1}k \right) = 2a_{l}\phi\left(\frac{k}{a_{l}+1}\right),
 \\ 2\phi(n) & = 2\phi\left( \frac{a_{l}}{a_{l}+1}k \right) = 2a_{l} \phi\left(\frac{k}{a_{l}+1}\right).
\end{align*}
A sequence of primes, up to $10^8$, satisfying the above hypotheses is 
\[
\{2,3,5,11,19,37,73,109,1459,2179,2917,4357,8713\}.
\]
For $k < \prod_j p_j$, surely $k$ is prime with $\prod_j p_j$. The product of the prime numbers in the above sequence is of the order of $6 \cdot 10^{26}$. In order to improve the upper bound on $k$ we slightly modify our argument, allowing to the sequence to be not necessarily increasing. We proceed as follows: 
\begin{itemize}
\item If $k$ is not divisible by $2 \cdot 3 \cdot 5 \cdot 11$, then we simply apply our argument with no changes.
\item If $2 \cdot 3 \cdot 5 \cdot 11 \cdot 7 | \, k$, we take the sequence starting with $2$, $3$, $5$, $11$, $7$ and then built using the rules as explained before. The solution to \eqref{equation:M:2} will be again $n=a_j/(a_j+1)k$, where $j$ is the smallest index such that $(p_j,k)=1$. The sequence continues as
\begin{align*}
\{&2, 3, 5, 11, 7, 13, 19, 29, 37, 41, 43, 59, 73, 83, 109, 113, 131, 163, 173, 181,\\ 
&227, 257, 331, 347, 353, 379, 419, 491, 523, 571, 601, 653, 661, 677, \dots 12011\}
\end{align*}
In this case, since the sequence contains $7$, it becomes dense enough in the primes to give a large upper bound. We found that such upper bound can be taken to be of the order of $2 \cdot 10^{310}$.
\item If $2 \cdot 3 \cdot 5 \cdot 11 | \, k$, $(7,k)=1$ and $(13,k)=1$, then $n = 36/{55} \, k$ is a solution of \eqref{equation:M:2}.
\item If $2 \cdot 3 \cdot 5 \cdot 11 \cdot 13| \, k$, $(7, k)=1$ and $(23,k)=1$, then $n = 66/{95} \, k$ is a solution of \eqref{equation:M:2}.
\item If $2 \cdot 3 \cdot 5 \cdot 11 \cdot 13 \cdot 23| \, k$ and $(7,k)=1$, we find, proceeding as before, the following sequence 
\begin{align*}
\{&2, 3, 5, 11, 13, 23, 19, 37, 73, 109, 131, 229, 457, 571, 1459, 1481, 2179, 2621,\\
&2917, 2963, 4357, 8713, 49921, 1318901, 3391489, 6782977, 13565953\},
\end{align*}
that gives the upper bound of $2\cdot 10^{83}$.
\end{itemize}
We conclude the proof noticing that our solution differs from the previously known solutions since in our case $n<k$, while in \cite{Hasanalizade2021} $n \ge k$.
\end{proof}

\begin{remark}
Our trick can be repeated arbitrarily as long as we can find particular solutions, in order to improve the upper bound on $k$. 
\end{remark}

As a consequence we get our main theorem.
\begin{proof}[Proof of Theorem \ref{main}]
Theorem \ref{odd:M:2} proves that for all odd $k \le 2 \cdot 10^{100}$ there are at least five solutions to equation \eqref{equation:M:2}, while Theorem \ref{even:M:2} proves that for all even $ k \le 4 \cdot 10^{58}$ we have at least three solutions. Overall, we get three solutions for all $k \le 4 \cdot 10^{58}$.
\end{proof}

\begin{remark}
We recall that while the upper bounds for our solutions are fairly high, increasing the upper bound of $4 \cdot 10^{58}$ for Hasanalizade's solution would require a lot of effort as pointed out by Holt in \cite{Holt2003}.
\end{remark}

\begin{center}
    \section{Some more special solutions}
\end{center}

We open the section with a modified version of first part of Lemma 1 from \cite{Hasanalizade2021} that removes the hypothesis $3| \, k$.

\begin{lemma}\label{lemma:modified}
Let $k$ be a positive odd integer, and $j$ such that $j$ and $2j +k$ have the same prime factors. Take $g=(j, 2j+k)$ and $r$ such that $2j/gr +1$ and $(2j+k)/gr+1$ are both prime and prime with $k$. Then
\[
n = 2j \left( \frac{2j+k}{g}r+1 \right)
\]
is a solution of equation \eqref{equation:M:2}.
\end{lemma}
\begin{proof}
We have that 
\begin{align*}
\phi(n+k) & = \phi\left(2j\left(\frac{2j+k}{g}r+1\right) + k\right) \\
    &= \phi \left( (2j+k) \left(\frac{2j}{g}r + 1\right) \right) \\
    &=2j\phi(2j+k) \frac{r}{g} \\
2\phi(n) & = 2\phi \left(2j\left(\frac{2j+k}{g}r+1\right)\right) \\
    &= 2\phi(2j) \phi\left(\frac{2j+k}{g}r+1\right) \\
    &= 2 \phi(2j)(2j+k) \frac{r}{g}.
\end{align*}
Now we conclude using that
\[
2j \phi(2j+k) = 2j \phi(2(2j+k)) = \phi(4j(2j+k)) =  2\phi(2j) (2j+k).
\]
\end{proof}
Motivated by explicit computations, we would like to find more general formulas that provide new families of solutions to \eqref{equation:M:2}. In practice, for $k \le 10^4$, we find always at least four solutions in the range $n \le 10^6$. The only $k$ that has exactly $4$ solutions is $k=6$, namely
\[
n \in \{4,6,7,10\}.
\] 
We increase the range of $n$ up to $10^8$ and we don't find more solutions. \\
We notice that $n=6$ is the solution given by Makowski and $n=10$ is the solution from Hasanalizade, while $n=4$ is obtained via our method with $a_l=2$. There is no known family of solutions providing $n=7$. In an attempt to find such a family, we prove the following result. 
\begin{proposition}
Assume that there exists $p$ prime such that $2p-1$ is prime, $(p,k)=1$, $(2p-1,k)=1$ and $(p-1)| \, k$. Then
\[
n= \frac{p}{p-1}k
\]
is a solution of \eqref{equation:M:2}.
\end{proposition}
\begin{proof}
It is immediate to check that
    \begin{align*}
 \phi(n+k) & = \phi\left( k\frac{p}{p-1} +k \right) = \phi\left( k\frac{2p-1}{p-1} \right) = (2p-2) \phi\left(\frac{k}{p-1} \right), \\
  2\phi(n) & = 2\phi\left( k\frac{p}{p-1} \right) = 2(p-1) \phi\left(\frac{k}{p-1} \right).
 \end{align*}
\end{proof}
For $k=6$, we can take $p=7$ and obtain the solution $n =7 $. Unfortunately, this cannot be generalized to all even $k$, since condition $(p-1)| \, k$ is satisfied only for a finite number of $p$ and might happen that for none of them the other requirements are met (e.g., $k=10$). 

We observe that for any odd $k$, thanks to the previous proposition, we can take $p=2$ and recover the solution $n=2k$ given by Makowski.

Finally, we conclude giving one more way of building solutions in another special case.

\begin{proposition}
Let $k$ be an even positive integer. Assume that there exists $m$ such that $m|k$, $(m+2,k)=1$, $(m+4,k)=1$ and $\phi (m+2)=\phi(m+4)$. Then
\[
n = \frac{m+4}{m}k
\]
is a solution of \eqref{equation:M:2}.
\end{proposition}
\begin{proof}
Again, it is immediate to check that

 \begin{align*}
 \phi(n+k) & = \phi\left( k\frac{m+4}{m} +k \right) = 2\phi\left( k\frac{m+2}{m} \right) = 2 \phi(m+2)\phi\left(\frac{k}{m} \right), \\
  2\phi(n) & = 2\phi\left( k\frac{m+4}{m} \right) = 2 \phi(m+4)\phi\left(\frac{k}{m} \right) .
 \end{align*}
\end{proof}

\bibliographystyle{plain}
\bibliography{bibtex}{}

\begin{thebibliography}{1}

\bibitem{GrahamHoltPomerance1999}
S.~W. Graham, J.~J. Holt, and C.~Pomerance.
\newblock On the solutions to $\phi(n) = \phi(n + k)$.
\newblock In Kálmán Györy, Henryk Iwaniec, and Jerzy Urbanowicz, editors,
  {\em Number Theory in Progress: Proceedings of the International Conference
  on Number Theory organized by the Stefan Banach International Mathematical
  Center in Honor of the 60th Birthday of Andrzej Schinzel, Zakopane, Poland,
  June 30-July 9, 1997}, pages 867--882. De Gruyter, 2012.

\bibitem{Hasanalizade2021}
E.~Hasanalizade.
\newblock On the equation {$\phi(n+k)=2\phi(n)$}.
\newblock {\em Integers}, 21:Paper No. A68, 2021.

\bibitem{Holt2003}
J.~J. Holt.
\newblock The minimal number of solutions to $\phi(n) = \phi(n + k)$.
\newblock {\em Mathematics of Computation}, 72(244):2059--2061, 2003.

\bibitem{Makowski1974}
A.~Makowski.
\newblock On the equation {$\phi(n+k)=2\phi(n)$}.
\newblock {\em Elem. Math.}, 29:13, 1974.

\bibitem{Schinzel1958}
A.~Schinzel.
\newblock Sur l'\'{e}quation {$\phi(x+k)=\phi(x)$}.
\newblock {\em Acta Arith.}, 4:181--184, 1958.

\bibitem{SchinzelWakulicz1959}
A.~Schinzel and A.~Wakulicz.
\newblock Sur l'\'{e}quation {$\phi(x+k)=\phi(x)$} \uppercase{II}.
\newblock {\em Acta Arith.}, 5:425--426, 1959.

\bibitem{Sierpinski1956}
W.~Sierpi\'{n}ski.
\newblock Sur une propri\'{e}t\'{e} de la fonction {$\phi(n)$}.
\newblock {\em Publ. Math. Debrecen}, 4:184--185, 1956.

\bibitem{Sungjing2021}
K.~Sungjin.
\newblock On the equations $\phi(n) = \phi(n + k)$ and $\phi(p-1) = \phi(q-1)$.
\newblock {\em International Journal of Number Theory}, 17(06):1287--1305,
  2021.

\end{thebibliography}

\end{document}